\documentclass[12pt]{amsart}
\usepackage{amsmath,amsfonts,euscript,amscd,amsthm,amssymb,upref,graphics,color,verbatim}
\usepackage[normalem]{ulem}

\numberwithin{equation}{section}

\newtheorem{theorem}{Theorem}[section]
\newtheorem{lemma}[theorem]{Lemma}
\newtheorem{proposition}[theorem]{Proposition}

\theoremstyle{definition}
\newtheorem{definition}[theorem]{Definition}

\theoremstyle{remark}
\newtheorem{remark}[theorem]{Remark}

\numberwithin{equation}{section}

\newcommand{\rien}[1]{}

\newcommand{\Q}{\ensuremath{\mathbb{Q}}}
\newcommand{\C}{\ensuremath{\mathbb{C}}}
\newcommand{\D}{\ensuremath{\mathbb{D}}}
\newcommand{\R}{\ensuremath{\mathbb{R}}}

\renewcommand{\epsilon}{\varepsilon}
\renewcommand{\phi}{\varphi}
\renewcommand{\emptyset}{\varnothing}

\begin{document}

\newpage

\renewcommand{\baselinestretch}{1.07}

\title[Parabolic cylinders]{Automorphisms of $\C^2$ with parabolic cylinders}

\author[L. Boc Thaler]{Luka Boc Thaler$^{\dag}$}

\author[F. Bracci]{Filippo Bracci$^{\dag\dag}$}
\author[H. Peters]{Han Peters}

\address{L. Boc Thaler: Faculty of Education, University of Ljubljana, SI--1000 Ljubljana, Slovenia.} \email{luka.boc@pef.uni-lj.si}
\address{F. Bracci: Dipartimento di Matematica, Universit\`a di Roma ``Tor Vergata", Via della Ricerca
Scientifica 1, 00133, Roma, Italia.} \email{fbracci@mat.uniroma2.it}

\address{ H. Peters: Korteweg de Vries Institute for Mathematics\\
University of Amsterdam\\
the Netherlands} \email{hanpeters77@gmail.com}

\thanks{$^{\dag}$  Supported by the SIR grant ``NEWHOLITE - New methods in holomorphic iteration'' no. RBSI14CFME and by the research program P1-0291 from ARRS, Republic of Slovenia}
\thanks{$^{\dag\dag}\,$Partially supported by the MIUR Excellence Department Project awarded to the
Department of Mathematics, University of Rome Tor Vergata, CUP E83C18000100006 and PRIN {\sl Real and Complex Manifolds: Topology, Geometry and holomorphic dynamics} n.2017JZ2SW5}


\vfuzz=2pt

\begin{abstract} A {\sl parabolic cylinder} is an invariant, non-recurrent Fatou component $\Omega$ of an automorphism $F$ of $\C^2$ satisfying: (1) The closure of the $\omega$-limit set of $F$ on $\Omega$ contains an isolated fixed point, (2) there exists a univalent map $\Phi$ from $\Omega$ into $\mathbb C^2$ conjugating $F$ to the translation $(z,w) \mapsto (z+1, w)$, and (3) every limit map of
 $\{F^{\circ n}\}$ on $\Omega$ has one-dimensional image. In this paper we  prove the existence of parabolic cylinders for an explicit class of maps, and show that examples in this class can be constructed as compositions of shears and overshears.
 \end{abstract}

\maketitle \vfuzz=2pt

\section{Introduction}

\subsection{Main result}

The description of Fatou components plays a central role in our understanding of holomorphic dynamical systems. In one complex variable the different kinds of Fatou components have been precisely characterized. In the last decade there has been significant progress in higher dimensions as well, but many fundamental questions remain unanswered. A primary objective is to describe invariant Fatou components in terms of the limit behavior of orbits, their complex geometry, and if possible, to give a normal form for the action of the map on the Fatou component.

The purpose of this paper is to shed more light on a specific kind of non-recurrent Fatou component, which we will call a {\sl parabolic cylinder}. Recall that an invariant Fatou component is called non-recurrent if every orbit on the component eventually departs from every compact subset of the component.

\begin{definition}\label{Def:parabolic-cyl}
Let $F$ be an automorphism of $\C^2$. An invariant non-recurrent Fatou component $\Omega$ is called a {\sl parabolic cylinder} if
\begin{enumerate}
\item the closure of the $\omega$-limit set of $F$ on $\Omega$ contains an isolated fixed point,
\item there exists a univalent map $\Phi:\Omega\to \C^2$, conjugating $F$ to the translation $$(z,w) \mapsto (z+1, w),$$
\item all limit maps of $F$ on $\Omega$ have dimension one.
\end{enumerate}
\end{definition}

Our main result is the following:

\begin{theorem}\label{main}
Let $F$ be an automorphism of $\C^2$ of the form
\begin{equation}\label{form-intro}
F(z,w)=\left(z+f(w)z^2+O(z^3, z^3w), e^{2\pi i \theta} w +g(w)z+O(z^2, z^2w)\right),
\end{equation}
where  $\theta\in\R\backslash\Q$ is Diophantine, $f(0)\neq 0$, and  $g(w)=O(w^2)$. Then there exists a parabolic cylinder $\Omega$ for $F$ that is biholomorphically equivalent to $\C^2$. Moreover, every limit map of $F$ on $\Omega$ has image $ \{0\}\times \C$.
\end{theorem}

The first example of an automorphism of $\mathbb C^2$ with a non-recurrent Fatou component on which all limit maps have rank $1$ was given in \cite{JL}. There an explicit map of the form
\begin{equation}\label{lilovexample}
G(z,w)=(z+z^2+O(z^3,z^4w,z^6w^2), w-\frac{z^2w}{2}+O(z^3,z^3w,z^3w^2)),
\end{equation}
was constructed, and it was shown that $G$ exhibits a Fatou component on which the orbits converge to the fixed plane $\{0\} \times \mathbb C$.  By composing $G$ with a rotation $(z,w) \mapsto (z,e^{2\pi i\theta} w)$, one obtains a non-recurrent Fatou component where the orbits converge to the rotating $w$-axis. Our main result implies that one can also obtain the normal form $(z,w) \mapsto (z+1, w)$.

We emphasize that the term $g(z)w$ that is allowed to appear in the second coordinate of the maps $F$ vanishes for the map $G$ given in equation \eqref{lilovexample}. The term of the form $g(z)w$ makes it significantly harder to prove convergence of iterates, and the proof given in \cite{JL} breaks down for the family considered here. In fact, it will be clear that our proof fails when $\theta = 0$. The requirement that $\theta \in \R \setminus \Q$ is Diophantine can likely be relaxed, but we chose it for convenience.

\subsection{Fatou components in several variables}

Let $F$ be a holomorphic self-map of a complex manifold $X$. The {\sl Fatou set} of $F$ is the set of points $p\in X$ for which there exists an open neighborhood $U\ni p$ such that the sequence of iterates $\{F^{\circ n}\}$ form a normal family on $U$. The connected components of the Fatou set of $F$ are called {\sl Fatou components} of $F$. A Fatou component $\Omega\subset X$ of $F$ is {\sl invariant} if $F(\Omega)=\Omega$. Following Bedford-Smillie \cite{BS}, an invariant Fatou component is called \emph{recurrent} if it contains a recurrent orbit, i.e. an orbit that accumulates at a point in $\Omega$. In a \emph{non-recurrent} invariant Fatou component all orbits eventually leave any compact subset.

For rational functions in one complex variable there is a complete description of all possible Fatou components and the dynamics of the map on such components is quite well understood. In particular, the invariant non-recurrent Fatou components are ``Leau-Fatou petals'' at a parabolic fixed point. All orbits in such petals converge to the fixed point and on such petals the map is conjugated to a translation  via the so-called ``Fatou coordinate''.

Despite significant recent progress, including the construction of wandering domains \cite{ABDPR, BB} and the classification of invariant Fatou components \cite{BS, LP}, the situation is not nearly as well understood in $\C^2$.

Let $F$ be an automorphism of $\C^2$. If $\Omega\subset\C^2$ is a Fatou component of $F$, we say that a holomorphic map $h:\Omega\to \C^2\cup\{\infty\}$ is a {\sl limit map} of $F$ on $\Omega$ if there exists a sequence $\{F^{\circ n_k}\}$ which converges uniformly on compacta of $\Omega$ to $h$ --- here, for the sake of uniformizing notation, we let $h\equiv \infty$ in case $\{F^{\circ n_k}\}$ compactly diverges to $\infty$.

\subsection{Fatou components of polynomial automorphisms.} If $F$ is a polynomial automorphism of $\mathbb C^2$, the Jacobian determinant $\delta$  is necessarily constant and different from $0$. When $|\delta|=1$ all Fatou components $\Omega$ of $F$  are recurrent, the so-called Siegel domains, and $h(\Omega) = \Omega$ for any limit map $h$. This does not complete the description, as it remains an open question whether $\Omega$ must be topologically trivial, see for example \cite{Bedford}.

In the case $|\delta| < 1$ the orbits in a recurrent Fatou component converge exponentially fast to either an attracting fixed point or to a $1$-dimensional properly embedded Riemann surface $\Sigma \subset \Omega$, see \cite{BS}. In the latter case, which could be called an \emph{attracting cylinder}, the action of $f$ on the invariant set $\Sigma$ is that of an irrational rotation, and $\Sigma$ is equivalent to either the disk or an annulus. Whether an annulus can actually occur is a pressing open question.

The non-recurrent case has been described in \cite{LP}, under the additional assumption $|\delta| < \frac{1}{\mathrm{deg}^2(f)}$. In this case all orbits converge to a parabolic-attracting fixed point, and the component is biholomorphic to $\mathbb C^2$, by a result of Ueda \cite{U0}, and $F$ is conjugate on $\Omega$ to a map $(z,w)\mapsto (z+1,w)$.

\subsection{Fatou components of holomorphic automorphisms.} Little is known about which other phenomena can occur when considering non-polynomial automorphisms of $\C^2$. An invariant Fatou component is called {\sl attracting} if  all the orbits in the component converge to the same (necessarily fixed) point  $p\in \C^2$. By \cite{PVW, RR}, a recurrent attracting Fatou component $\Omega$ is necessarily biholomorphic to $\C^2$, the spectrum of $dF_p$ is contained in the (open) unit disk and $F$ is conjugate to a polynomial triangular map on $\Omega$.

In \cite{BRS} (see also \cite{Rep} for the construction of multiple ``petals'') the authors constructed an attracting non-recurrent Fatou component biholomorphic to $\C\times \C^\ast$, where the map is semi-conjugate to a translation over $\C$. It is an open question whether all attracting Fatou components in $\mathbb C^2$ are conjugate to either $\mathbb C^2$ or $\mathbb C \times \C^\ast$.

Contrary to the polynomial case, there are known to exist holomorphic automorphisms with non-recurrent Fatou components that are not attracting. Examples of such maps were given in \cite{JL}, including the map $G$ mentioned in equation \eqref{lilovexample}. Here we prove the existence of these non-recurrent Fatou components for a considerably larger class of maps, prove these examples are all biholomorphic to $\mathbb C^2$, and construct Fatou coordinates, showing that the dynamics is conjugate to a translation. We note that in \cite{Rep2}, J. Reppekus, exploiting the example in \cite{BRS} and blowing-up, shows that there exist parabolic cylinders that are not biholomorphically equivalent to $\C\times \C^\ast$. Whether $\mathbb C^2$ and $\C\times \C^\ast$ are the only two possibilities is again an open question.

Another natural open question, partially addressed in \cite{JL, LP}, concerns the uniqueness of limit sets, i.e. whether all limit maps must have the same image.

\subsection{Outline of the paper.}

In section (2) we introduce several coordinate changes, defined only locally near the invariant $w$-axis. In section (3), Proposition \ref{h} we show that the map $F$ is locally conjugate to a map $H$ of the form
$$
H(u,w)=\left(u+1+\frac{A}{u}+O\left(\frac{1}{u^2}\right),\lambda w +O\left(\frac{1}{u^2}\right)\right).
$$
This simpler form is exploited in Proposition \ref{prop1} to prove the existence of a non-recurrent Fatou component on which the orbits converge to the $w$-axis. In section (4), Proposition \ref{prop2} we prove the existence of the Fatou components, which implies that the Fatou component is biholomorphically equivalent to $\mathbb C^2$.

In section (5) we show that maps $F$ satisfying the constraints in the theorem can actually be constructed as finite compositions of shears and overshears. Recall that such compositions form a dense subset of all automorphisms of $\C^2$ in the compact-open topology \cite{AL}. We note that $F$ cannot have constant Jacobian determinant, and therefore cannot be approximated by compositions of shears only.

\subsection*{Acknowledgments} In a first version of this paper, ``parabolic cylinders'' were named ``non-recurrent Siegel cylinders''. However, the term ``parabolic cylinders'' seems to be more appropriate, due to Property (2) in Definition~\ref{Def:parabolic-cyl}. We thank Eric Bedford for stimulating discussions about this and other facts related to the paper. We also thank the referee for  very useful comments which improved much the original paper. In particular, for finding a mistake in the original version of Lemma~\ref{lemma:sequence}.

\section{Preliminaries}\label{prelim}

 \textcolor{black}{ 
In this section we introduce various maps, together with useful estimates, that will serve in the next sections to prove our main result.
 }

\medskip

Throughout this paper we will use $\lambda=e^{2\pi i \theta}$ where $\theta\in\R\backslash\Q$ is {\sl diophantine}, i.e. there exist $c,r>0$ such that $|\lambda^n-1|\geq c n^{-r}$ for every $n\geq1$. Such numbers form a dense subset of the unit circle with full measure. Note that if $\lambda$ is diophantine then $\lambda^{-1}$ is also diophantine and satisfies the same estimates.
\medskip

We will be using the following notation. Let $u(x)$ and $v(x)$ be two functions. By writing $u(x)=O(v(x))$ we mean that there exist a constant $C>0$ such that $|u(x)|\leq C|v(x)|$ for all $x$ in a neighborhood of the origin where $u$ and $v$ are defined.  The notation $u(x)=o(v(x))$ as $x\rightarrow a$ means that $u(x)/v(x)\rightarrow 0$ as $x\rightarrow a$. In case of a sequence $u_n$ of complex numbers, the notation $u_n=O(v(n))$ has to be understood as $|u_n|\leq C|v(n)|$  for all $n\in \mathbb N$.

\medskip

The following was used in \cite{ABP}, we repeat the proof for convenience of the reader.

\begin{lemma}\label{lemma:lambda} There exist constants $C,r>0$ such that for every integer $n\geq1$ and for every $ m\geq0$,
$$
\left|\sum_{j=m}^{\infty}\lambda^{ jn}\right|<Cn^r.
$$
\end{lemma}
\begin{proof}
Let $N\geq m$. Since $\lambda=e^{2\pi i \theta}$ where $\theta\in\mathbb{R}\backslash\mathbb{Q}$ is diophantine, there exist $c,r>0$ such that $|\lambda^n-1|\geq cn^{-r}$ for all $n$. This gives the bound
$$
\left|\sum_{j=m}^{N}\lambda^{ nj}\right|  = \left|\sum_{j=m}^{N}\frac{\lambda^{ n(j+1)}-\lambda^{ nj}}{\lambda^{n}-1}\right|=
\left|\frac{1}{\lambda^{ n}-1}\sum_{j=m}^{N}(\lambda^{ n(j+1)}-\lambda^{ nj})\right|<\left|\frac{2}{\lambda^{ n}-1}\right|<Cn^r,
$$
and we are done.
\end{proof}

 \textcolor{black}{ 
 Let $\gamma:[0,\infty)\rightarrow [0,\infty)$ be an increasing function that satisfies $\gamma(2x)\geq 2\gamma(x)$ for all $x\geq 0$. For $R>0$  and $\delta>0$ we let
\begin{equation}\label{setK}
K_{R,\delta}:=\overline{\{u\in \C \mid Re(u)>-\gamma(\delta) \text{ and }  \text{arg}(u-R)\in[-3\pi/4,3\pi/4]\}},
\end{equation}
and
\[
U_{R,\delta}:=\{(u,w)\in \mathbb{C}^2\mid  u\in K_{R,\delta} \text{ and } |w|< \delta \}.
\]
To be precise, both the set $K_{R,\delta}$ and $U_{R,\delta}$ depend on the choice of the function $\gamma$, so that one should more appropriately name them $K_{R,\delta, \gamma}$ and $U_{R,\delta, \gamma}$. However, in order not to burden notation, and since no confusion should arise, we avoid mentioning $\gamma$ in this definition.
} 

Note that $U_{R_1,\delta_1}\subset U_{R_2,\delta_2}$ when
$R_2\leq R_1$ and $\delta_1\leq\delta_2$.

\begin{lemma}\label{lemma:sequence} Let $R>0$ \textcolor{black}{and $\delta>0$}. There exist $r>0$ and $\tilde{C}>0$ such that for every integer $n\geq1$, $m\geq0$ and \textcolor{black}{$u\in K_{R,\delta}$},
\begin{equation}\label{estimate}
\left|\sum_{j=m}^{\infty}\frac{\lambda^{nj}}{u+j}\right|<\frac{\tilde{C} n^r}{|u+m|}.
\end{equation}
In particular, the series $\sum_{j=0}^{\infty}\frac{\lambda^{nj}}{u+j}$ is converging uniformly on compacta of \textcolor{black}{$K_{R,\delta}$} for every $n\geq 1$.
\end{lemma}

\begin{proof}
Let $N\geq m$. Lemma \ref{lemma:lambda} gives
\begin{equation*}
\begin{aligned}
\left|\sum_{j=m}^{N}\frac{\lambda^{nj}}{u+j}\right| & = \left|\frac{1}{u+N}\sum_{j=m}^{N}\lambda^{nj}-\sum_{j=m}^{N-1}\left(\frac{1}{u+j+1}-\frac{1}{u+j}\right)\sum_{k=m}^{j}\lambda^{nk}\right|\\
& <  \frac{Cn^r}{|u+N|}+Cn^r\cdot \sum_{j=m}^{N-1}\frac{1}{|(u+j+1)(u+j)|}<\frac{\tilde{C} n^r}{|u+m|},
\end{aligned}
\end{equation*}
with the constant $\tilde{C}$  chosen to be independent from $N$ and $u$, and we are done. \textcolor{black}{Note that here we have used the fact that for every $u\in K_{R,\delta}$ and for every $m\geq 0$ we have $|u+m|<|u+k|$ for all $k>M:=\max\{2\gamma(\delta),m\}$, hence there is a constant $C'(m, \gamma(\delta))$ such that $$\frac{1}{\min_{k\geq m}{ |u+k|}}=\frac{1}{\min_{m\leq k \leq M}{ |u+k|}}\leq \frac{C'}{|u+m|}$$ for all $u\in K_{R,\delta}$. }
\end{proof}

\begin{lemma}\label{Phi} Let $g:\C \to \C$ be an entire function such that $g(0)=g'(0)=0$. Let $g(w)=\sum_{\ell=2}^{\infty}d_{\ell}w^{\ell}$ be its expansion at $0$. Then for every $\delta>0$, there exists $R=R(\delta)>0$, which depends continuously on $\delta$, such that the map
$$\Phi(u,w):=\left(u,w+\lambda^{-1}\sum_{\ell=2}^{\infty}\left(d_{\ell}w^{\ell}\sum_{k=0}^{\infty}\frac{\lambda^{(\ell-1)k}}{u+k} \right)\right)$$
is univalent on $U_{R,\delta}$. Moreover $\Phi(u,w)=\left(u,w+O\left(\frac{1}{u}\right)\right)$ and for every $\delta>0$ and $R>0$ there exists  $R'\geq R$ such that $\Phi$ is univalent on $U_{R',\delta}$ and $\Phi(U_{R',\delta})\subset U_{R,2\delta}$.
Also, for every $\delta>0$ there exists $R''\geq R(\delta)$ such that $U_{R'',\delta/2}\subset \Phi(U_{R(\delta),\delta})$.
\end{lemma}

\begin{proof} Let  \textcolor{black}{$R,\delta>0$}. By Lemma \ref{lemma:sequence} the map $\Phi(u,w)$ is a well defined holomorphic map on  \textcolor{black}{$\{(u,w)\in \mathbb{C}^2\mid u\in K_{R,\delta} \}$} and $\Phi(u,w)=\left(u,w+O\left(\frac{1}{u}\right)\right)$. In order to check  injectivity, first observe that $\Phi(u,w)=\Phi(u',w')$ implies $u=u'$. Therefore
$\Phi(u,w)=\Phi(u',w')$ if and only if
$$w-w'+\lambda^{-1}\sum_{\ell=2}^{\infty}d_{\ell}(w^{\ell}-w'^{\ell})\sum_{k=0}^{\infty}\frac{\lambda^{(\ell-1)k}}{u+k}=0.$$
Assuming that $w\neq w'$ we can divide this equation by $w-w'$ to obtain
$$1+\lambda^{-1}\sum_{\ell=2}^{\infty}d_{\ell}\left(\frac{w^{\ell}-w'^{\ell}}{w-w'}\right)\sum_{k=0}^{\infty}\frac{\lambda^{(\ell-1)k}}{u+k}=0.$$
Since $\left|\frac{w^{\ell}-w'^{\ell}}{w-w'}\right|\leq \ell\delta^{\ell-1}$,  taking into account Lemma \ref{lemma:sequence} and that $g$ is entire, we can choose $R$ large enough so that
$$\left|\sum_{\ell=2}^{\infty}d_{\ell}\left(\frac{w^{\ell}-w'^{\ell}}{w-w'}\right)\sum_{k=0}^{\infty}\frac{\lambda^{(\ell-1)k}}{u+k}\right|<\frac{\tilde{C}}{|u|}\sum_{\ell=2}^{\infty}|d_{\ell}|\delta^{\ell-1}\ell(\ell-1)^{r}<1$$
everywhere on $U_{R,\delta}$, and hence $\Phi(u,w)=\Phi(u',w')$ implies $(u,w)=(u',w')$. This last inequality also implies that $R$ depends continuously on $\delta$.

In order to prove the final statement, let
\begin{equation}\label{h(u,w)}
h(u,w):=\lambda^{-1}\sum_{\ell=2}^{\infty}\left(d_{\ell}w^{\ell}\sum_{k=0}^{\infty}\frac{\lambda^{(\ell-1)k}}{u+k} \right)=O\left(\frac{1}{u}\right),
\end{equation}
and let $R''\geq R(\delta)$ be such that $|h(u,w)|<\frac{\delta}{4}$ for  \textcolor{black}{$u\in K_{R'',\frac{\delta}{2}}$} and $|w|\leq \delta$. Let \textcolor{black}{$u_0\in K_{R'',\frac{\delta}{2}}$} and let $|w_0|\leq \delta/2$. Let $C:=\{w\in \C: |w-w_0|=\delta/2\}$. By the triangular inequality $|w|\leq \delta$ for all $w\in C$. Thus, for all $w\in C$,
\[
|w-w_0|=\frac{\delta}{2}>|h(u_0, w)|.
\]
Hence, by the Rouch\'e theorem, the functions $w\mapsto w-w_0$ and $w\mapsto w+h(u_0,w)-w_0$ have the same number of zeros in $\{w\in \C: |w-w_0|<\delta/2\}$. In particular, there exists $w_1\in \C$ such that $|w_1|<\delta$ and $w_1+h(u_0,w_1)=w_0$. Therefore, $(u_0, w_0)\in \Phi(U_{R(\delta), \delta})$. By the arbitrariness of $(u_0, w_0)$, this proves that $U_{R'', \delta/2}\subset \Phi(U_{R(\delta), \delta})$.
\textcolor{black}{Finally it follows from \eqref{h(u,w)} that for sufficiently large $R'>R(\delta)$ we have $|h(u,w)|<\frac{1}{2}$ on $\overline{U_{R',\delta}}$, hence $\Phi(U_{R',\delta})\subset U_{R,2\delta}$.}

\end{proof}

\begin{lemma}\label{psi}  Let $f:\C \to \C$ be an entire function such that $f(0)=1$. Let $f(w)=1+\sum_{\ell=1}^{\infty}d_{\ell}w^{\ell}$ be its expansion at $0$. The map $\Psi:\C^2\to \C^2$ defined as
$$
\Psi(u,w):=\left(u+\sum_{\ell=1}^{\infty}\frac{d_{\ell}}{\lambda^{\ell}-1}w^{\ell},w \right)
$$
 is a holomorphic automorphism of $\mathbb{C}^2$. \textcolor{black}{Furthermore, let $C,r>0$ be as in Lemma \ref{lemma:lambda} and define
 \begin{equation}\label{good-lambda}
 \gamma(\delta):=C\sum_{\ell=1}^{\infty}|d_\ell|\ell^r\delta^\ell.
 \end{equation}
  Then $\gamma:[0,\infty)\rightarrow[0,\infty)$ is an increasing function satisfying $\gamma(2\delta)\geq2\gamma(\delta)$ for all $\delta\geq 0$.} Moreover,  for every $\delta>0$ we have $||\Psi(u,w)-(u,w)||<\gamma(\delta)$ for all $(u,w)\in \mathbb{C}\times\{w\in \C: |w|\leq \delta\}$. \textcolor{black}{Finally,  using \eqref{good-lambda} in the definition of $K_{R,\delta}$ (see \eqref{setK} ), we also have that}, for every $\delta>0$ there exists $M_\delta>0$ such that for all $R\geq M_\delta$
\textcolor{black}{\[
\Psi(U_{2R,\delta})\subset U_{R,2\delta}.
\]}
\end{lemma}
\begin{proof}
Clearly the series $\sum_{\ell=1}^{\infty}d_{\ell}w^{\ell}$ is absolutely convergent on compacta of  $\mathbb{C}$. Recall that by our assumption $\lambda$ satisfies the condition $| \lambda^n-1|>cn^{-r}$, for some $c,r>0$. It follows that
\begin{equation*}
\left|\sum_{\ell=1}^{\infty}\frac{d_{\ell}}{\lambda^{\ell}-1}w^{\ell}\right|<C\sum_{\ell=1}^{\infty}|d_\ell|\ell^r|w|^\ell.
\end{equation*}
From this last inequality we can deduce that the series which appears in the first coordinate of the map $\Psi$ is absolutely convergent on compacta in $\C$. Therefore $\Psi$ is an automorphism of $\C^2$.

\textcolor{black}{From the previous considerations and the definition of $\gamma$, it follows immediately that  $||\Psi(u,w)-(u,w)||<\gamma(\delta)$ for all $(u,w)\in \mathbb{C}\times\{w\in \C: |w|\leq \delta\}$.}

\textcolor{black}{In order to prove the last statement, let us write $(u_1,w_1)=\Psi(u_0,w_0)$ and observe that $w_1=w_0$. Since $(u_0,w_0)\in U_{R,\delta}$ it follows that ${\sf Re}(u_0)>-\gamma(\delta)$ which  implies  $${\sf Re}(u_1)={\sf Re}(u_0)-\gamma(\delta)>-2\gamma(\delta)>-\gamma(2\delta).$$
From here one can easily deduce that there exists $M_\delta>0$, such that for all $R>M_\delta$ we have $\Psi(U_{2R,\delta})\subset U_{R,2\delta}$.}
\end{proof}

\begin{lemma}\label{tau} Let $f:\C \to \C$ be an entire function with $f(0)=0$. Let $f(w)=\sum_{\ell=1}^{\infty}d_{\ell}w^{\ell}$ be its expansion at $0$.
Then for every $\delta>0$, there exists $R>0$ such that the map
$$\tau(u,w):=\left(u- \sum_{k=0}^{\infty}\frac{f(\lambda^kw)}{u+k},w\right)$$
is univalent on $U_{R,\delta}$. Moreover $R$ depends continuously on $\delta$ and
$\tau(u,w)=\left(u+O\left(\frac{1}{u}\right),w \right)$. In particular, for every $\delta>0$ and $R>0$ there exists $R'\geq R$ such that $\tau$ is univalent on $U_{R',\delta}$ and \textcolor{black} {$\tau(U_{R',\delta})\subset U_{R,2\delta}$}.
Also, for every $\delta>0$ there exists $R''\geq R(\delta)$ such that \textcolor{black} {$U_{R'',\delta/2}\subset \tau(U_{R(\delta),\delta})$}.

\end{lemma}
\begin{proof} We first prove that the map is well defined, {\sl i.e.}, we prove that for every  \textcolor{black}{$R,\delta>0$} the series
\begin{equation}\label{prva}
\sum_{k=0}^{\infty}\frac{f(\lambda^kw)}{u+k}
\end{equation}
converges uniformly on compacta of \textcolor{black}{$\{(u,w)\in \mathbb{C}^2\mid u\in K_{R,\delta} \}$}.
By Lemma \ref{lemma:sequence},
\begin{equation*}
\left|\sum_{k=0}^{N}\frac{f(\lambda^kw)}{u+k}\right|=\left|\sum_{k=0}^{N}\sum_{j=1}^{\infty}\frac{\lambda^{kj}d_jw^j}{u+k}\right|
=\left|\sum_{j=1}^{\infty}d_jw^j\sum_{k=0}^{N}\frac{\lambda^{kj}}{u+k}\right|
<\frac{\tilde C}{|u|}\sum_{j=1}^{\infty}|d_j||w|^jj^r.
\end{equation*}
Since $f$ is entire, the last series converges uniformly on compacta, and so does the series~\eqref{prva}.

In order prove injectivity, we first observe that $\tau(u,w)=\tau(u',w')$ implies that $w=w'$. If $u\neq u'$ then $\tau(u,w)=\tau(u',w')$  if and only if
$$u-u'-\sum_{k=0}^{\infty}\frac{f(\lambda^kw)}{u+k}-\frac{f(\lambda^kw)}{u'+k}=0.$$
Dividing this equation by $u-u'$ we obtain
\begin{equation}\label{Eq:inter-1}
1+\sum_{k=0}^{\infty}\frac{f(\lambda^kw)}{(u+k)(u'+k)}=0.
\end{equation}
Given $\delta>0$, we can find $R$ large enough such that for all \textcolor{black}{$u,u'\in K_{R,\delta}$} we have
 $$\sum_{k=0}^{\infty}\frac{1}{|(u+k)(u'+k)|}<\frac{1}{\sup_{|w|<\delta}|f|}.$$
Therefore, \eqref{Eq:inter-1} cannot be satisfied in $U_{R,\delta}$, and hence $\tau$ is  injective in $U_{R,\delta}$.

By the previous considerations it follows also that $\tau(u,w)=\left(u+O\left(\frac{1}{u}\right),w \right)$.

The last statement follows by applying Rouch\'e's theorem as in  Lemma \ref{Phi}.
\end{proof}

\textcolor{black}{
\begin{remark} It is worth noticing that, except for Lemma~\ref{psi} where we need to choose a suitable $\gamma$ in the definition of the set $K_{R,\delta}$, for the other lemmas, any choice of $\gamma$ works well.
\end{remark}}

\section{Non-recurrent Fatou component}
Let $F$ be a holomorphic automorphism of $\C^2$ of the form
\begin{equation}\label{form}
F(z,w)=\left(z+f(w)z^2+O(z^3),\lambda w +g(w)z+O(z^2)\right),
\end{equation}
where $f$ and $g$ are entire functions in $\C$, $f(0)\neq 0$ and $g(w)=O(w^2)$. Notice that the inverse $F^{-1}$ has the same form as $F$. From now on we assume without loss of generality that $f(0)=1$, since otherwise we can simply conjugate $F$ with a dilatation in the first factor.

\textcolor{black}{The aim of this section is to show that  $F$  has an invariant non-recurrent Fatou component $\Omega$ with $\omega$-limit set $\{0\}\times \C \subset \partial\Omega$ (Proposition~\ref{prop1}). To achieve this, we use the maps introduced in the previous section to change coordinates and get estimates allowing to show that $F$ maps
sets of the type ``a small sector times a disc of radius $\delta$'' (which, in the $(u,w)$ coordinates, correspond to the sets $U_{R,\delta}$  introduced in the previous section)
into sets of the same type. Moreover, we prove that for each orbit of $F$ on such sets, the first variable moves toward $0$, while the modulo of the second variable stays close to the modulo of the second variable of the starting point of the orbit. The Fatou component $\Omega$ is then defined as the Fatou component of $F$ which contains the union of all such (suitably chosen) ``small sectors times discs'' sets.     }
\medskip

\textcolor{black}{Let $f(w)=1+\sum_{\ell=1}^{\infty}d_{\ell}w^{\ell}$ be the expansion of $f$ at $0$, and let $C>1$ such that
$$\left|\sum_{\ell=1}^{\infty}\frac{d_{\ell}}{\lambda^{\ell}-1}w^{\ell}\right|<C\sum_{\ell=1}^{\infty}|d_\ell|\ell^r|w|^\ell.$$ Define
$$\gamma(\delta):=C\sum_{\ell=1}^{\infty}|d_\ell|\ell^r\delta^\ell,$$
and observe that $\sup_{|w|\leq\delta}|f(w)-1|\leq \gamma(\delta)$ and $\gamma(2\delta)>2\gamma(\delta)$ for all $\delta>0$.}

\medskip

\textcolor{black}{\noindent{\sl Assumption:} From now on, without further mentioning it, we use the above function $\gamma$  in the definition of the sets $K_{R,\delta}$ (see \eqref{setK}) and  $U_{R,\delta}$.}

\medskip

\textcolor{black}{We define the biholomorphic map $\Theta:\C^*\times \C\to \C^*\times \C$ as
\begin{equation}\label{Theta}
\Theta(u,w):=(-\frac{1}{u},w).
\end{equation}
Since  $F$ is an automorphism of $\C^2$ which leaves invariant the affine line $\{0\}\times\C$, it follows that the map
$$\tilde{F}:= \Theta^{-1}\circ F\circ\Theta$$
is a well defined biholomorphic map $\tilde{F}:\C^*\times \C\to \C^*\times \C$. A quick computation shows
$$
\tilde{F}(u,w)=\left(u+f(w)+O\left(\frac{1}{u}\right), \lambda w-\frac{g(w)}{u}+O\left(\frac{1}{u^2}\right)  \right)
.$$}
Let  $\delta'>0$  and observe \textcolor{black}{that the set $U_{R',\delta'}\subset\C^*\times\C$ for all $R'>0$. Therefore, $\tilde{F}$
 is well defined and univalent on $U_{R',\delta'}$}.

 Moreover, given  $R''>0$, we can find  $R'>0$ such that $\tilde F(U_{R',\delta'})\subset U_{R'',2\delta'}$.
 \textcolor{black}{In order to see this, let $(u_1,w_1)=\tilde F(u_0,w_0)$. Since $|w_1|=|w_0|+O(1/u
 _0)$ it is easy to see that for sufficiently large $R'$ we obtain $|w_1|<2\delta'$. Also, observe that
 \begin{equation*}
 \begin{split}
  {\sf Re}(u_1)&={\sf Re}(u_0)+1+{\sf Re}(f(w)-1)+{\sf Re}(O\left(\frac{1}{u}\right))\\&>-2\gamma(\delta')+1-O\left(\frac{1}{u}\right)>-\gamma(2\delta'),
 \end{split}
 \end{equation*}
 where the last inequality holds for all sufficiently large $R'$.}

 Let $g(w)=\sum_{\ell=2}^{\infty}b_{\ell}w^{\ell}$ be the expansion of $g$ at $0$ and  let $\Phi$ be as in Lemma \ref{Phi}.

 Fix $\delta>0$. By Lemma \ref{Phi}, there exists $R''>0$  such that $\Phi^{-1}$ is well defined and univalent on $U_{R'',4\delta}$. By the previous considerations, there exists  $R'\geq R'_0(\delta)$ such that $\tilde F(U_{R',2\delta})\subset U_{R'',4\delta}$ and finally, by Lemma \ref{Phi}, there exists $R>0$ such that $\Phi$ is univalent on $U_{R, \delta}$ and $\Phi(U_{R,\delta})\subset U_{R',2\delta}$. Thus,
 \[
 G:=\Phi^{-1}\circ \tilde{F}\circ\Phi
 \]
 is well defined and univalent on $U_{R,\delta}$.

\begin{lemma}\label{G-form}  For $(u,w)\in U_{R,\delta}$, we have
\[
G(u,w)=\left(u+f(w)+O\left(\frac{1}{u}\right),\lambda w+O\left(\frac{1}{u^2}\right) \right).
\]
\end{lemma}

\begin{proof} Let us write $(u_1,w_1):=G(u,w)$. First observe that
$$(\tilde{F}\circ\Phi)(u,w)=\left(u+f(w)+O\left(\frac{1}{u}\right),\lambda w+\sum_{\ell=2}^{\infty}b_{\ell}w^{\ell}\sum_{k=0}^{\infty}\frac{\lambda^{(\ell-1)k}}{u+k}-\frac{1}{u}\sum_{\ell=2}^{\infty}b_{\ell}w^{\ell} +O\left(\frac{1}{u^2}\right)  \right).$$
Since
$$\Phi^{-1}(u,w)=\left(u,w-\lambda^{-1}\sum_{\ell=2}^{\infty}b_{\ell}w^{\ell}\sum_{k=0}^{\infty}\frac{\lambda^{(\ell-1)k}}{u+k}+O\left(\frac{1}{u^2}\right) \right),$$
it follows that
\begin{align*}
u_1&=u+f(w)+O\left(\frac{1}{u}\right)\\
w_1&=\lambda w+\sum_{\ell=2}^{\infty}b_{\ell}w^{\ell}\sum_{k=0}^{\infty}\frac{\lambda^{(\ell-1)k}}{u+k}-\frac{1}{u}\sum_{\ell=2}^{\infty}b_{\ell}w^{\ell}-\lambda^{-1}\sum_{\ell=2}^{\infty}b_{\ell}\lambda^{\ell}w^{\ell}\sum_{k=0}^{\infty}\frac{\lambda^{(\ell-1)k}}{u+k+f(w)} \\
&= \lambda w+\sum_{\ell=2}^{\infty}b_{\ell}w^{\ell}\left(-\frac{1}{u}+\sum_{k=0}^{\infty}\frac{\lambda^{(\ell-1)k}}{u+k}-\sum_{k=1}^{\infty}\frac{\lambda^{(\ell-1)k}}{u+k+(f(w)-1)+O\left(\frac{1}{u}\right)}\right)+O\left(\frac{1}{u^2}\right)\\
&=\lambda w+\sum_{\ell=2}^{\infty}b_{\ell}w^{\ell}
\sum_{k=1}^{\infty}\frac{\lambda^{(\ell-1)k}(f(w)-1)+O\left(\frac{1}{u}\right)}{(u+k)\left(u+k+(f(w)-1)+O\left(\frac{1}{u}\right)\right)}+O\left(\frac{1}{u^2}\right)\\
&=\lambda w+(f(w)-1)\sum_{\ell=2}^{\infty}b_{\ell}w^{\ell}
\sum_{k=1}^{\infty}\frac{\lambda^{(\ell-1)k}}{(u+k)\left(u+k+(f(w)-1)+O\left(\frac{1}{u}\right)\right)}+O\left(\frac{1}{u^2}\right)\\
&=\lambda w +O\left(\frac{1}{u^2}\right).
\end{align*}
The last equality follows from the fact that
$$\left|\sum_{k=1}^{\infty}\frac{\lambda^{(\ell-1)k}}{(u+k)\left(u+k+(f(w)-1)+O\left(\frac{1}{u}\right)\right)}\right|<\frac{C(\ell-1)^r}{|u^2|}$$
on $U_{R,\delta}$ for some $C>0$, which follows similarly as Lemma \ref{lemma:sequence}.
\end{proof}

Recall that $f(z)=1+\sum_{k=1}^{\infty}d_kw^k$ and let
 $$
\Psi(u,w)=\left(u+\sum_{k=1}^{\infty}\frac{d_k}{\lambda^k-1}w^k,w\right)
$$
 be the map defined in Lemma \ref{psi}. \textcolor{black}{Observe that $\left|\sum_{k=1}^{\infty}\frac{d_k}{\lambda^k-1}w^k\right|\leq \gamma(\delta)$ for all for all $|w|\leq\delta$. Therefore,
by Lemma \ref{psi}, there exists $R'>R$ such that $\Psi(U_{R',\delta/2})\subset U_{R, \delta}$. This implies that
\[
\tilde{G}:=\Psi^{-1}\circ G\circ \Psi
\]
is a well defined univalent map on $U_{R',\delta/2}$.}

In order to avoid burdening notations, we still denote $R'$ by $R$ \textcolor{black}{and $\frac{\delta}{2}$ by $\delta$}, so that $\tilde G$ is univalent on $U_{R,\delta}$.

\textcolor{black}{\begin{lemma}Given $\delta>0$, we can choose $R$ sufficiently large so  that for $(u,w)\in U_{R,\delta}$ we have
\begin{equation}\label{tildeGn}
\tilde G(u,w)=\left(u+1+O\left(\frac{1}{u}\right), \lambda w+O\left(\frac{1}{u^2}\right)\right).
\end{equation}
\end{lemma}
\begin{proof} Let us write $(u',w')=\tilde{G}(u,w)$. By  Lemma \ref{psi} the function $q(w):=\sum_{k=1}^{\infty}\frac{d_k}{\lambda^k-1}w^k$ is an entire function, therefore  we can choose $R$  so large that $|q(w)|<R$ for all $|w|\leq \delta$.
It follows that
\[
w'=\lambda w +O\left(\frac{1}{u^2}\right).
\]
In the first coordinate we get
\begin{equation*}
\begin{split}
u'&=u+\sum_{k=1}^{\infty}\frac{d_k}{\lambda^k-1}w^k +f(w)-\sum_{k=1}^{\infty}\frac{d_k}{\lambda^k-1}\left(\lambda w+O\left(\frac{1}{u^2}\right)\right)^k\\
&+O\left(\frac{1}{u+q(w)}\right)+O\left(\frac{1}{u^2}\right).
\end{split}
\end{equation*}
Next observe that $O\left(\frac{1}{u+q(w)}\right)=O\left(\frac{1}{u}\right)$. Hence,
\begin{align*}
u'&=u+\sum_{k=1}^{\infty}\frac{d_k}{\lambda^k-1}w^k +f(w)-\sum_{k=1}^{\infty}\frac{d_k}{\lambda^k-1}\left(\lambda w+O\left(\frac{1}{u^2}\right)\right)^k+O\left(\frac{1}{u}\right)\\
&=u+1+\sum_{k=1}^{\infty}\frac{d_k}{\lambda^k-1}w^k +\sum_{k=1}^{\infty}d_k w^k-\sum_{k=1}^{\infty}\frac{d_k}{\lambda^k-1}\left(\lambda w+O\left(\frac{1}{u^2}\right)\right)^k+O\left(\frac{1}{u}\right)\\
&=u+1+\sum_{k=1}^{\infty}\frac{d_k}{\lambda^k-1}\left(\lambda^kw^k -\left(\lambda w+O\left(\frac{1}{u^2}\right)\right)^k\right)+O\left(\frac{1}{u}\right)\\
&=u+1+O\left(\frac{1}{u}\right),
\end{align*}
where we used the fact that $\sum_{k=1}^{\infty}\frac{d_k}{\lambda^k-1}\left(\lambda^kw^k -\left(\lambda w+O\left(\frac{1}{u^2}\right)\right)^k\right)=O\left(\frac{1}{u^2}\right)$.
\end{proof}}

\begin{remark}\label{Rem:UtildeG}
It follows from \eqref{tildeGn} that for any $R''>0$ there exists $R'\geq R$ such that  $\tilde G(U_{R',\delta})\subset U_{R'',2\delta}$.
\end{remark}
\textcolor{black}{By previous lemma, we can write
$$\tilde G(u,w)=\left(u+1+\frac{h(w)}{u}, \lambda w\right)+O\left(\frac{1}{u^2}\right),$$
for some holomorphic function
$h:\{w\in\C:|w|\leq\delta\}\to\C$. Let us denote $A:=h(0)$ and let}
\begin{equation}\label{Eq:Tau}
\tau(u,w)=\left(u- \sum_{k=0}^{\infty}\frac{h(\lambda^kw)-A}{u+k},w\right).
\end{equation}
By Lemma \ref{tau}, there exists $R''>0$ such that $\tau^{-1}$ is well defined on $U_{R'',2\delta}$. By Remark~\ref{Rem:UtildeG} we can choose  $R'\geq R$ such that  $\tilde G(U_{R',\delta})\subset U_{R'',2\delta}$. Finally, we can choose  $R_1>0$ such that $\tau$ is well defined and univalent on \textcolor{black}{$U_{R_1,\delta/2}$ and $\tau(U_{R_1,\delta/2})\subset U_{R',\delta}$}.

Once again, in order to simplify notation, we will write $R$ instead of $R_1$ \textcolor{black}{and $\delta$ instead of $\frac{\delta}{2}$}, so that
\begin{equation}\label{Hj}
H(u,w):=\tau^{-1}\circ \tilde{G}\circ \tau
\end{equation}
is well defined and univalent on  $U_{R,\delta}$.

\begin{proposition}\label{h} For $(u,w)\in U_{R, \delta}$, we have $H(u,w)=\left(u+1+\frac{A}{u}+O\left(\frac{1}{u^2}\right),\lambda w +O\left(\frac{1}{u^2}\right)\right)$.
\end{proposition}
\begin{proof}
By Lemma \ref{tau},  $\tau(u,w)=(u+O\left(\frac{1}{u}\right),w)$. This, together with \eqref{tildeGn}, implies
$$\tilde{G} \circ \tau(u,w)=\left(u+1+\frac{h(w)}{u}-\sum_{k=0}^{\infty}\frac{h(\lambda^kw)-A}{u+k}+O\left(\frac{1}{u^2}\right),\lambda w+ O\left(\frac{1}{u^2}\right)\right).
$$
Next observe that
$$\tau^{-1}(u,w)=\left(u+ \sum_{k=0}^{\infty}\frac{h(\lambda^kw)-A}{u+k}+O\left(\frac{1}{u^2}\right),w\right).$$
Let us write $(u',w')=H(u,w)$ and observe that $w'=\lambda w+ O\left(\frac{1}{u^2}\right)$. Let us write $w'=\lambda w+\alpha(u)$ where $\alpha(u)=O\left(\frac{1}{u^2}\right)$ is a holomorphic function (with coefficients depending on $w$).  In the first coordinate we get
\begin{align*}
u'&= u+1+\frac{h(w)}{u}-\sum_{k=0}^{\infty}\frac{h(\lambda^kw)-A}{u+k}
+\sum_{k=0}^{\infty}\frac{h(\lambda^{k+1}w+\lambda^{k}\alpha(u))-A}{u+k+1+O\left(\frac{1}{u}\right)}+O\left(\frac{1}{u^2}\right)\\
&=u+1+\frac{A}{u}-\sum_{k=0}^{\infty}\frac{h(\lambda^kw)-A}{u+k}
+\sum_{k=1}^{\infty}\frac{h(\lambda^{k}w+\lambda^{k-1}\alpha(u))-A}{u+k+O\left(\frac{1}{u}\right)}+O\left(\frac{1}{u^2}\right)\\
&=u+1+\frac{A}{u}+ \sum_{k=1}^{\infty}\frac{h(\lambda^{k}w+\lambda^{k-1}\alpha(u))-h(\lambda^{k}w)}{u+k+O\left(\frac{1}{u}\right)}
\\&-\sum_{k=1}^{\infty}\frac{h(\lambda^{k}w)O\left(\frac{1}{u}\right)}{(u+k+O\left(\frac{1}{u}\right))(u+k)}+O\left(\frac{1}{u^2}\right).
\end{align*}
We are going to show that both of the infinite sums in the above expression are of order $O\left(\frac{1}{u^2}\right)$.

Clearly
\begin{equation*}
\begin{split}
&\left|\sum_{k=1}^{\infty}\frac{h(\lambda^{k}w)O\left(\frac{1}{u}\right)}{(u+k+O\left(\frac{1}{u}\right))(u+k)}\right|\\&\leq \sup_{|w|\leq\delta}|h(w)|\cdot O\left(\frac{1}{|u|}\right)\cdot\sum_{k=1}^{\infty}\frac{1}{|(u+k+O\left(\frac{1}{u}\right))(u+k)|}=O\left(\frac{1}{|u^2|}\right).
\end{split}
\end{equation*}

As before we can write convergent power series:
$$
h(w+\lambda^{-1}\alpha(u))-h(w)=\sum_{\ell=1}^\infty \sum_{j=0}^\infty \frac{(\ell+j)!\lambda^{-\ell}}{j!}b_{j+\ell}w^j (\alpha(u))^{\ell},
$$
where $h(w)=\sum_{j=0}^\infty b_{j}w^j$. Now observe that the same computation as in Lemma \ref{lemma:sequence} tells us that the sum
$$
\sum_{k=0}^{\infty}\frac{\lambda^{k(j+\ell)}}{u+k+O\left(\frac{1}{u}\right)}
$$
converges and that
$$\left|\sum_{k=0}^{\infty}\frac{\lambda^{k(j+\ell)}}{u+k+O\left(\frac{1}{u}\right)}\right|\leq \frac{C(j+\ell)^r}{|u|}.$$
 Therefore
\begin{align*}
\left|\sum_{k=0}^{\infty}h(\lambda^kw+\lambda^{k-1}\alpha(u))-h(\lambda^kw)\right|&= \left|\sum_{\ell=1}^\infty \sum_{j=0}^\infty \frac{(\ell+j)!\lambda^{-\ell}}{j!}b_{j+\ell}w^j (\alpha(u))^{\ell}\sum_{k=0}^{\infty}\frac{\lambda^{k(j+\ell)}}{u+k+O\left(\frac{1}{u}\right)}\right|\\
&=\frac{C}{|u|}\sum_{\ell=1}^\infty \sum_{j=0}^\infty \frac{(\ell+j)!(j+\ell)^r}{j!}\left|b_{j+\ell}w^j (\alpha(u))^{\ell}\right|\\
&=O\left(\frac{1}{u^2}\right).
\end{align*}
\end{proof}

For given $(u_0, w_0)\in U_{R,\delta}$ and integer $n\geq 1$ we write $(u_{n},w_{n}):=H^n(u_0,w_0)$.

\begin{lemma}\label{induction}Given $\delta>0$ there exists $T_\delta\geq R$ such that for every $T\geq T_\delta$ and $(u_0,w_0)\in U_{T, \frac{\delta}{2}}$, we have
\[(u_{n},w_{n})\in U_{T +\frac{n}{2},\delta}
\]
 for every $n\geq 1$. Moreover given $0<\varepsilon\leq\frac{\delta}{2}$, there exists  $T_\varepsilon \geq T_\delta$ such that
\[
 |w_{n}-\lambda^{n} w_0|<\varepsilon
 \]
for every $(u_0,w_0)\in U_{T_\varepsilon, \delta/2}$ and for every $n\geq 1$.
\end{lemma}
\begin{proof} Fix $0<\varepsilon\leq \delta/2$. By Proposition \ref{h} we can choose $T\geq R$ and $C>0$ such that $|u_{1}-u_0-1|<\frac{1}{2}$ and $|w_{1}-\lambda w_0|<\frac{C}{|u_0^2|}$ on $U_{T,\delta}$ and such that
 $\sum_{n=0}^{\infty}\frac{C}{(T+\frac{n}{2})^2}<\varepsilon$. Using induction it is easy to see  that ${\sf Re}(u_{n})>T+\frac{n}{2}$ and $|w_{n}-\lambda^nw_0|<\varepsilon$.
\end{proof}

\begin{lemma}\label{estimate2} Let $\delta>0$ and $T\geq T_\delta$ be as in  Lemma \ref{induction}. For every compact subset $K\subset U_{T, \frac{\delta}{2}}$  there exists a constant $C>0$ so that for every  $(u_0,w_0)\in K$ and every $n\geq 1$ we have
$$ \frac{1}{|u_{n}|}\leq \frac{C_1}{n} \quad \text{and}\quad \left|\frac{1}{u_{n}}-\frac{1}{n}\right|\leq C\frac{\log n}{n^2}.$$
\end{lemma}

\begin{proof} The first inequality follows directly from  ${\sf Re}(u_{n})>T+\frac{n}{2}$. As for the second inequality first observe that
\begin{equation}\label{ineq1}
\left|\frac{1}{u_{n}}-\frac{1}{n}\right|\leq \frac{C_1|u_{n}-n|}{n^2}.
\end{equation}

Recall that $(u',w'):=H(u,w)=(u+1+\frac{A}{u},\lambda w) +O(1/u^2)$ where $\left|u'-u+1+\frac{A}{u}\right|\leq\frac{C'}{|u|^2}$ for some $C'>0$. Using the  inequality \eqref{ineq1} we can now deduce that
\begin{align*}
|u_{n}-n|&\leq\left|u_0+n+A\left(\frac{1}{u_0}+\ldots+\frac{1}{u_{n-1}}\right)-n\right|+C'C_1\sum_{k=1}^{n}\frac{1}{k^2}\\
&\leq|u_0|+|A|C_1\sum_{k=1}^{n}\frac{1}{k}+C'C_1\sum_{k=1}^{n}\frac{1}{k^2}\\
&\leq |u_0|+  C_2 \log n + C_3
\end{align*}
where $C_2,C_3>0$. Finally we obtain
$$\left|\frac{1}{u_{n}}-\frac{1}{n}\right|\leq \frac{C_1(|u_0|+  C_2 \log n + C_3)}{n^2}=O \left(\frac{\log n}{n^2}\right).$$

\end{proof}

\begin{proposition}\label{prop1}Let $F$ be an automorphism of the form \eqref{form}. Then $F$  has an invariant non-recurrent Fatou component $\Omega$ with $\omega$-limit set
$\{0\}\times \C \subset \partial\Omega$.
\end{proposition}

\begin{proof}  Let $\delta>0$ and $R>0$. Let $\Theta$ be the map defined in \eqref{Theta}. Let  $V_{R,\delta}:=\Theta(U_{R,\delta})$ and observe that
\textcolor{black}{$V_{R,\delta}=\Lambda_{R,\delta}\times\D_{\delta}$, where $\Lambda_{R,\delta}:=\{z\in\C\mid -\frac{1}{z}\in K_{R,\delta}\}$ and $\mathbb{D}_{\delta}:=\{w\in \C: |w|<\delta\}$.}

\smallskip

\textcolor{black}{We divide the proof in four steps. First, we prove that for every $\delta>0$ one can find a suitable $R(\delta)>0$ so that $\{F^{\circ n}\}$ is a normal family on $V_{R(\delta),\delta/4}$. Next, we prove that the image of every limit map of $\{F^{\circ n}\}$ on such a set is contained in the affine line $\{z=0\}$ and always contains a disc of radius $\delta/16$. Then, we show that every limit map of the Fatou component $\Omega$ of $F$ which contains $V_{R(\delta),\delta/4}$ (for $\delta\to +\infty$), has image $\{0\}\times \C$. Finally, we show that $\Omega$ is non-recurrent. }
\smallskip

\noindent{\sl Step 1: For every $\delta>0$ there exists $R(\delta)>0$ such that $\{F^{\circ n}\}$ is a normal family on $V_{R(\delta),\delta/4}$}.

 Let $\delta>0$. Let $\tau$ be the map defined in \eqref{Eq:Tau} and let $\Phi$ be the one defined in Lemma~\ref{Phi}.

Let $M_\delta>0$ be given by Lemma \ref{psi}. By Lemma \ref{Phi}, Lemma \ref{psi} and Lemma \ref{tau},  for any $T\geq M_\delta$ there exists $R(\delta, T)>0$ such that
 $$\overline{\tau^{-1}\circ \Psi^{-1}\circ\Phi^{-1}\circ\Theta^{-1}(V_{R(\delta,T), \delta/4})}\subset\textcolor{black}{\Theta^{-1}(V_{2T, 2\delta})=  U_{2T,2\delta}}$$
 and
 \textcolor{black}{
\[
\Phi\circ\Psi\circ\tau(U_{2T,4\delta})\subset U_{T,32\delta},
\]}
 for every $n\geq 1$.

Therefore, if \textcolor{black}{  $T_{2\delta}>0$ is given by Lemma \ref{induction}, take $T:=\max\{M_{2\delta},T_{2\delta}\}$}  and $R(\delta):=R(\delta ,T)$. Since
\begin{equation}\label{EqFn}
F^{\circ n}=\Theta\circ\Phi\circ\Psi\circ\tau\circ H^n\circ \tau^{-1}\circ \Psi^{-1}\circ\Phi^{-1}\circ\Theta^{-1},
\end{equation}
  Lemma \ref{induction} implies at once that the family  $\{F^{\circ n}\}$ is normal on $V_{R(\delta), \delta/4}$.

  \smallskip

  \noindent{\sl Step 2: Let $g$ be a limit map of $\{F^{\circ n}\}$  on $V_{R(\delta), \delta/4}$. Then
 \[
  \{0\}\times\mathbb{D}_{\frac{\delta}{16}}\subset g(V_{R(\delta), \delta/4})\subset \{0\}\times \C.
  \]}

Let $\varepsilon=\delta/16$. Let $T_{\varepsilon}$ be as in Lemma  \ref{induction}, and let $R \geq \max\{M_\delta,T_\varepsilon\}$. Denote by $\pi_j:\C^2 \to \C$, $j=1,2$ the projection on the $j$-th component, that is, $\pi_1(z,w)=z$, $\pi_2(z,w)=w$. It follows from Lemma  \ref{induction} that  for all $(z,w)\in V_{R,\delta/4}$ and $n\geq 1$:
\[
|\pi_2\circ \Psi\circ\tau\circ H^n\circ \tau^{-1}\circ \Psi^{-1}\circ\Phi^{-1}\circ\Theta^{-1}(z,w)-\lambda \pi_2\circ \Phi^{-1}\circ \Theta^{-1}(z,w)|<\frac{\delta}{16}.
\]
Taking into account that, by Lemma \ref{Phi}, $\pi_2\circ \Phi(u,w)$ and $\pi_2\circ \Phi^{-1}(u,w)$  are of the form $w+O(1/u)$, the previous equation and \eqref{EqFn} imply that there exist $R'\geq R$ such that
\begin{equation}\label{stima2F}
|\pi_2 \circ F^{\circ n}(z,w)-\lambda^n w|< \frac{\delta}{10}
\end{equation}
for all $(z,w)\in V_{R',\delta/4}$ and $n\geq 1$.

Moreover, again by Lemma \ref{induction}, given $\eta>0$, there exists $n_0$ such that for all $n\geq n_0$,\begin{equation}\label{stima1F}
|\pi_1 \circ F^{\circ n}(z,w)|\leq \eta
\end{equation}
for all $(z,w)\in V_{R',\delta/4}$.

Let $\{n_j\}$ be any increasing sequence for which $F^{\circ n_j}$ converges uniformly on compacta of $V_{R(\delta), \delta/4}$ to a holomorphic function $g$ and  $\lambda^{n_j}$ converges to some $\mu\in \partial \D$.  It follows from  \eqref{stima1F} that  $g(z,w)=(0, g_2(z,w))$ for all $(z,w)\in V_{R(\delta), \delta/4}$, where $g_2:V_{R(\delta), \delta/4}\to \C$ is holomorphic. Moreover, by \eqref{stima2F},
\begin{equation}\label{Eqg2}
|g_2(z,w)-\mu w|<\frac{\delta}{10}
\end{equation}
for all $(z,w)\in V_{R',\delta/4}$.

Fix $w_0\in \C$, $|w_0|<\delta/16$. Let $z_0\in \C$ be such that  \textcolor{black}{$-1/z_0\in K_{R',\frac{\delta}{16}}$}. Hence,  $(z_0,w_0)\in V_{R',\delta/4}$. Moreover, $\{z_0\}\times \{w\in\C : |w-w_0|<\delta/8\}\subset V_{R',\delta/4}$ and \eqref{Eqg2} holds on $|w-w_0|=\delta/8$. Therefore, by Rouch\'e's theorem, $\overline{\mu}g_2(z_0,w)-\overline{\mu}w_0$ and $w-\overline{\mu}w_0$ have the same number of zeros in $\{w\in\C : |w-w_0|<\delta/8\}$. Since $|\overline{\mu}w_0-w_0|\leq 2 |w_0|<\delta/8$, it follows that there exists $w_1$, with $|w_1-w_0|<\delta/16$ such that $\overline{\mu}g_2(z_0,w_1)=w_0$. By the arbitrariness of $w_0$, it follows that $\D_{\delta/16}\subset g_2(V_{R(\delta), \delta/4})$, which completes step 2.

\smallskip

\noindent{\sl Step 3: There exists an invariant Fatou component $\Omega$ such that  the image of any limit map of $\{F^{\circ n}|_\Omega\}$ is $\{0\}\times \C$}.

Let $\{\delta_m\}$ be an increasing sequence of positive real numbers which converges to $+\infty$. We can choose $R(\delta_{m+1})\geq R(\delta_m)$ for all $m\geq 0$. Let $V_m:=V_{R(\delta_m), \delta_m/4}$, $m\geq 0$. Hence,
\textcolor{black}{
\[
\left\{(z,w)\in V_{m+1}: |w|< \frac{\delta_m}{4} \text{ and } {\sf Re}(-\frac{1}{z})>-\gamma\left(\frac{\delta_m}{4}\right) \right\}\subset V_m.
\] }
Therefore $V:=\cup_{m\geq 0} V_m$ is open and connected. Since $\{F^{\circ n}\}$ is a normal family on $V_m$ for all $m\geq 0$ by Step 1, the previous equation and a diagonal argument imply that  $\{F^{\circ n}\}$ is a normal family on $V$ and, hence,   $\{F^{\circ n}\}$ is normal on
\begin{equation}\label{Vset}
\mathcal{V}:=\bigcup_{n=0}^{\infty}F^{\circ n}(V).
\end{equation}

By \eqref{stima1F} and \eqref{stima2F}, $F^{\circ n}(V)\cap V\neq \emptyset$ for every $n\geq 1$.  Hence $\mathcal{V}$ is a $F$-forward invariant, open, connected set on which $\{F^{\circ n}\}$ is normal. Therefore, there exists an invariant Fatou component $\Omega$ which contains $\mathcal V$.

Now, let $g$ be a  limit map of $\{F^{\circ n}\}$ on $\Omega$. Hence, $g|_{V_m}$ is a limit map of $\{F^{\circ n}|_{V_m}\}$ for all $m\geq 0$. Then, it follows from  Step 2 that   $g(V)=\{0\}\times \C$. Since $V$ is open in $\mathcal V$, it follows as well that $g(\Omega)=\{0\}\times \C$.

\smallskip

\noindent{\sl Step 4:  $\Omega$ is non-recurrent}.

Observe that $F^{\circ n}(0,w)=(0,\lambda^n w)$ for all $w\in \C$. Equation \eqref{form} therefore implies that
$$
\frac{\partial^2 \pi_1(F^{\circ n})}{\partial z^2}(0,w)=2\sum_{k=0}^{n-1}f(\lambda^k w)=2\sum_{k=0}^{n-1}(f(\lambda^k w)-1)+2n.
$$

Let us first observe that $\sum_{k=0}^{n-1}(f(\lambda^k w)-1)$ is uniformly bounded in $w$ with respect to $n$. We have
\begin{align*}\left|\sum_{k=0}^{n-1}(f(\lambda^k w)-1)\right| &=\left|\sum_{k=0}^{n-1}\sum_{\ell=1}^{\infty}d_\ell\lambda^{k\ell} w^\ell \right|\\
&=\left|\sum_{\ell=1}^{\infty}d_\ell w^\ell\sum_{k=0}^{n-1}\lambda^{k\ell}  \right|\\
&\leq \sum_{\ell=1}^{\infty}|d_\ell| |w^\ell|\left|\sum_{k=0}^{n-1}\lambda^{k\ell}  \right|\\
&\leq C \sum_{\ell=1}^{\infty}|d_\ell|\ell^r |w^\ell|.
\end{align*}
Since $f(w)$ is an entire function, the last sum converges on uniformly on compacta of $\C$. This implies that
\[
\lim_{n\to\infty}\left| \frac{\partial^2 \pi_1(F^{\circ n})}{\partial z^2}(0,w) \right|=\infty.
\]
Therefore,   $(0,w)$ cannot be contain in any Fatou component of $F$ for all $w\in \C$. Thus $\{0\}\times \C\subset \partial \Omega$, which completes the proof.
\end{proof}

\begin{remark}\label{Vsupset}
The set $\mathcal V$ defined in \eqref{Vset} depends on the sequence $\{\delta_n\}$ and on the choice of $\{R(\delta_n)\}$. In particular, given any $\eta>0$ one can construct the open set $\mathcal V$ of the form \eqref{Vset} such that
\begin{equation}\label{Vsup}
\sup\{|z|: (z,w)\in \mathcal V\}<\eta.
\end{equation}
Indeed, unrolling the definition of $\mathcal V$, we see that
\[
\mathcal V=\cup_{n\geq 0}\cup_{m\geq 0}F^{\circ n}(V_m),
\]
where $V_m=V_{R(\delta_m), \delta_m/4}$, with $\{\delta_m\}$  an increasing sequence converging to $\infty$ and $R(\delta_m)$  a suitable increasing sequence of real positive numbers. By construction one can replace $R(\delta_m)$ with any $R_m\geq R(\delta_m)$ so that $\{R_m\}$ is still increasing. If we choose $R_m$ sufficiently large, by \eqref{stima1F}, $|\pi_1(F^{\circ n}(z,w))|<\eta$ for all $(z,w)\in F^{\circ n}(V_m)$ for all $n\geq 0$ and $m\geq 0$. Hence the corresponding set $\mathcal V$ satisfies \eqref{Vsup}.
\end{remark}

\section{Fatou coordinates}

Let $F$ be an automorphism of $\C^2$ of the form \eqref{form}. Let $\Omega$ be the invariant non-recurrent Fatou component of $F$ defined in  Proposition \ref{prop1}.

\textcolor{black}{
In this section we prove that there exists a global change of  holomorphic  coordinates on $\Omega$---the ``Fatou coordinates''---which makes $\Omega$ biholomorphic to $\C^2$, and so that $F$ has the form $(z,w)\mapsto (z+1,\lambda w)$ in this new coordinates. In order to make such a construction, we first show that $\Omega$ coincides with the set of points whose orbits pointwise accumulate  to $\{0\}\times \C$.  Then we define ``local approximating Fatou coordinates'' and show that they do converge to the  Fatou coordinates.}

Let $\{\delta_m\}$ be an increasing  sequence of positive real numbers converging to $\infty$ and let $\{R_m\}$ be an increasing sequence of positive real numbers such that $R_m\geq R(\delta_m)$ (where the $R(\delta_m)$'s are defined in Step 1 of the proof of Proposition \ref{prop1}). Let $V_m:=V_{R_m, \delta_m/4}$ and let
\[
\mathcal V:=\cup_{n\geq 0}\cup_{m\geq 0}F^{\circ n}(V_m).
\]
Let
$$
\mathcal{W^\iota}:=\bigcup_{n=0}^{\infty}(F^{-1})^{\circ n}(\mathcal V).
$$
Note that $\mathcal{W^\iota}$ is an open connected set such that $F(\mathcal{W^\iota})=\mathcal{W^\iota}$, and, since $\mathcal V\subset \Omega$, it follows that $\mathcal{W^\iota}\subseteq \Omega$.

\begin{lemma}\label{fatou}
Let $(z_0,w_0)\in \C^2\setminus(\{0\}\times \C)$ be such that \textcolor{black}{$\lim_{n\to \infty}\pi_1(F^{\circ n}(z_0,w_0))=0$ and $\{\pi_2(F^{\circ n}(z_0,w_0))\}$ is bounded, where, as before, $\pi_j$ is the projection on the $j$-th coordinate, $j=1,2$}.
Then for every set $\mathcal V$ as before, there exists $n_0=n_0(\mathcal V,z_0,w_0)$ such that $F^{\circ n}(z_0,w_0)\in \mathcal V$ for all $n\geq n_0$. In particular,  $\Omega=\mathcal{W}^{\iota}$.
\end{lemma}
\begin{proof}
If $F^{\circ n_0}(z,w)\in \mathcal V$ for some $n_0$, then $F^{\circ n}(z,w)\in \mathcal V$ for all $n\geq n_0$ since $F(\mathcal V)\subset \mathcal V$ by construction.

Therefore, we assume by contradiction that $F^{\circ n}(z,w)\not\in \mathcal V$ for all $n\geq 0$.

As we already notice, $F^{-1}$ has the same form of $F$, that is,
\[
F^{-1}(z,w)=(z+\tilde f(w)z^2+O(z^3), \overline{\lambda}w+\tilde g(w)z+O(z^2)),
\]
where $\tilde f, \tilde g:\C \to \C$ are holomorphic, $\tilde f(0)=-1$ (since we assumed $f(0)=1$) and $g(w)=O(w^2)$. Let $\chi(z,w)=(-z,w)$. The automorphism $\chi \circ F^{-1}\circ \chi$ has the same form as $F^{-1}$, but the coefficient of $z^2$ in the first coordinate is $1$.
Hence, by Proposition~\ref{prop1}, there exists an invariant non-recurrent Fatou component $\Omega^-$ of $\chi\circ F^{-1}\circ \chi$ with $\{0\}\times \C\subset \partial \Omega^-$, and $\Omega^-$ contains a connected open set $\tilde {\mathcal V}^-$ of the same form as \eqref{Vset}. Moreover, by Remark~\ref{Vsupset}, we can assume that $|z|<|z_0|$ for all $(z,w)\in \tilde {\mathcal V}^-$.

In particular, $\chi(\Omega^-)$ is a non-recurrent Fatou component  of $ F^{-1}$ with $\{0\}\times \C$ on the boundary,   contains the open set $\mathcal V^-:=\chi(\tilde {\mathcal V}^-)$, and $|z|<|z_0|$ for all $(z,w)\in  {\mathcal V}^-$, that is, $(z_0,w_0)\not\in \mathcal V^-$.

By the very definition of $\mathcal V$ and $\mathcal V^-$, it follows that $\mathcal V\cup  \mathcal V^-\cup (\{0\}\times \C)$ is a neighborhood of $\{0\}\times \C$. Therefore, since \textcolor{black}{ $\lim_{n\to \infty}\pi_1(F^{\circ n}(z_0,w_0))=0$,  $\{\pi_2(F^{\circ n}(z_0,w_0))\}$ is bounded } and $F^{\circ n}(z_0,w_0)\not\in \mathcal V$, the sequence $\{F^{\circ n}(z,w)\}$ has to be eventually  contained in $ \mathcal V^-$.

However, since  $F^{-1}(\mathcal V^-)\subset \mathcal V^-$, it follows that the entire orbit $\{F^{\circ n}(z,w)\}_{n\in \mathbb Z}$ is contained in $\mathcal V^-$, hence $(z_0,w_0)\in\mathcal V^-$, a contradiction.
\end{proof}

For natural numbers $n+1> j\geq 1$, we let
$$Q_n(u,w)=(u-n-A\log n,\lambda^{-n}w)$$
and let
$$
\varphi_n=Q_n\circ H^n,
$$
where the $H$ is defined in \eqref{Hj}.

For $n>0$ one can easily verify the following equality
\begin{equation}\label{funct1}
\varphi_{n}\circ H=\chi_{n}\circ\varphi_{n+1},
\end{equation}
where $\chi_{n}(u,w)=(u+1 +A\log(1+\frac{1}{n}), \lambda w)$.

\begin{lemma}\label{limit} For every $\delta>0$ there exists $S_\delta>0$ such that  the sequence $\{\varphi_n\}_{n\in \mathbb N}$ converges  uniformly on compacta of $U_{S_\delta,\delta/4}$ to a univalent map $\varphi:U_{S_\delta,\delta/4}\to \C^2$ such that
\begin{equation}\label{funct2}
\varphi\circ H=\chi\circ\varphi,
\end{equation}
where $\chi(u,w)=(u+1, \lambda w)$ and
\begin{equation}\label{made}
\varphi(u,w)=(u-A\log(u)+o(1), w+o(1))
\end{equation}
as ${\sf Re}(u)\rightarrow \infty$. Moreover, given any increasing sequence $\{\delta_m\}$ of positive real numbers converging to $\infty$, $\varphi:\bigcup_{m\geq 0}U_{S_{\delta_m}, \delta_m/4}\to \C^2$ is univalent.
\end{lemma}
\begin{proof} Fix $\delta>0$ and let $T_\delta$ be given by Lemma \ref{induction}. Let $(u_0,w_0)\in U_{T_\delta,\delta/2}$ and set $(u_n,w_n):=H^n(u_0,w_0)$. By Lemma \ref{induction} we have $(u_n,w_n)\in U_{T_\delta,\delta}$ for all $n$. Hence, by Proposition \ref{h},
\[
(u_{n+1},w_{n+1})=\left(u_n+1+\frac{A}{u_n}+O\left(\frac{1}{u_n^2}\right),\lambda w_n+O\left(\frac{1}{u_n^2}\right)\right),
\]
where the bounds in the $O$'s are uniform in $n$.
 Let us write
 \[
 (u',w')=\varphi_{n+1}(u_0,w_0)-\varphi_n(u_0,w_0).
 \]
 Observe that
\begin{align*}
w'&=\lambda^{-(n+1)}\left(\lambda w_n+O\left(\frac{1}{u_n^2}\right)\right)-\lambda^{-n}w_n=O\left(\frac{1}{u_n^2}\right)
\end{align*}
and
\begin{align*}
u'&=u_n+1+\frac{A}{u_n}+O\left(\frac{1}{u_n^2}\right)-(n+1)-A\log(n+1)-u_n+n+A\log n\\
&=A\left(\frac{1}{u_n}-\frac{1}{n}\right)+ A\left(\frac{1}{n}-\log(1+\frac{1}{n})\right) +O\left(\frac{1}{u_n^2}\right).
\end{align*}
By Lemma \ref{estimate2}  we have
$$\left|\frac{1}{u_n}-\frac{1}{n}\right|=O\left(\frac{\log n}{n^2}\right)$$
and
$$O\left(\frac{1}{|u_n|^2}\right)=O\left(\frac{1}{n^2}\right).$$
Next observe that $\left(\frac{1}{n}-\log(1+\frac{1}{n})\right)=O(\frac{1}{n^2})$. Since all bounds are uniform on compact subsets of $U_{T,\delta/2}$ and independent from $n$
it follows that $\sum_{n=j}^{\infty}(\varphi_{n+1}-\varphi_{n})$ converges absolutely, hence the sequence $\{\varphi_n\}$ converges uniformly on compacta of $U_{T_\delta,\delta/2}$ to a map $\varphi$.

Fix $\epsilon>0$. By Lemma \ref{induction}, $(u_n,w_n)\in U_{T_{\frac{\delta}{2}}+\frac{n}{2}, \delta/2}$ for every $(u_0,w_0)\in U_{T_{\delta/2}, \delta/4}$. Therefore, by the same lemma, there exists $n_0$ such that for all $n\geq n_0$ and all $(u,w)\in U_{T_{\delta/2}, \delta/4}$,
\[
|\varphi(u,w)-\varphi_n(u,w)|<\epsilon.
\]
Since all maps $\varphi^n$ are univalent it follows that $\varphi$ is also univalent on $U_{T_{\delta/2}, \delta/4}$. Setting $S_\delta:=T_{\delta/2}$ we have the first result.

Since $\varphi_{n}(u,w)=\sum_{k=0}^{n-1}(\varphi_{k+1}(u,w)-\varphi_k(u,w))$, \eqref{made} follows immediately from the previous computations. The functional equation \eqref{funct2} follows from \eqref{funct1} passing to the limit.
Finally observe that given any increasing sequence $\{\delta_m\}$ of positive real numbers converging to $\infty$ there exist a sequence of positive real numbers $\{S_{\delta_m} \}$ so that $\varphi$ is univalent on $\bigcup_{m\geq 0}U_{S_{\delta_m}, \delta_m/4}$.
\end{proof}

Let $\{\delta_m\}$ be an increasing sequence of positive real numbers converging to $\infty$ and let $\{S_{\delta_m}\}$ be the sequence given by Lemma \ref{limit}. Let $\{R_m\}$ be an increasing sequence of positive real numbers such that $R_m\geq \max\{R(\delta_m), S_{\delta_m}\}$ (where, as before, the $R(\delta_m)$'s are defined in Step 1 of the proof of Proposition \ref{prop1}). Let $V_m:=V_{R_m, \delta_m/4}$.

We define
\[
P=(\Theta\circ\Phi\circ\Psi\circ\tau\circ \varphi^{-1})^{-1}.
\]
By Lemma \ref{limit}, $P$ is a  univalent map defined on  $\bigcup_{m=1}^{\infty}V_{m}$.

By \eqref{EqFn} and \eqref{funct2},
\begin{equation}\label{funct3}
F=P^{-1} \circ
\chi\circ P
\end{equation}
on $\bigcup_{m=1}^{\infty}V_{m}$ for all $n\geq 0$.

\begin{proposition}\label{prop2} The Fatou component $\Omega$ is biholomorphic to $\C^2$ and there exists a  univalent map $\mathcal{Q}$ defined on $\Omega$ such that
\begin{equation}\label{funct4}
\mathcal{Q}\circ F= \chi\circ \mathcal{Q}
\end{equation}
where $\chi(u,w)=(u+1,\lambda w)$.
\end{proposition}
\begin{proof}
Let $V_m$ as before. Let $\mathcal V:=\cup_{n\geq 0}\cup_{m\geq 0}F^{\circ n}(V_m)$. By Lemma \ref{fatou}, $\Omega=\cup_{k=0}^{\infty} F^{-k}(\mathcal V)$.

We extend $P$ to a univalent map $\mathcal{Q}$ defined on $\Omega$ as follows. If $(z,w)\in \Omega$, there exists a natural number $n$ such that $F^{\circ n}(z,w)\in V_m$ for some $m$. Hence, we set
$$
\mathcal{Q}(z,w)=(\chi^{-1})^{\circ n}\circ P\circ F^{\circ n}(z,w).
$$
By \eqref{funct3}, this definition is well posed and $\mathcal Q:\Omega\to \C^2$ is univalent. The functional equation \eqref{funct4} therefore follows from \eqref{funct3}.

Now we prove that $\mathcal{Q}(\Omega)=\C^2$. Let $\Omega_n:=\cup_{k=0}^n (F^{-1})^{\circ n}(\mathcal V)$. Observe that
$$\mathcal Q(\Omega)=\cup_{n=0}^{\infty} (\chi^{-1})^{\circ n}\circ P\circ F^{\circ n}(\Omega_n)= \cup_{n=0}^{\infty} (\chi^{-1})^{\circ n}\circ P(\mathcal{V}).$$

From the definition of the maps $\Theta$, $\Phi$, $\Psi$ and $\tau$  and the set $\mathcal{V}$  we can find a sequence of $\rho_n\rightarrow\infty$ satisfying
$$\cup_{k=0}^{\infty}\{u\in \C:\text{Re}(u)>\rho_k\}\times\mathbb{D}_k\subset\tau^{-1}\circ\Psi^{-1}\circ\Phi^{-1}\circ\Theta^{-1}(\mathcal{V}).$$

It follows
$$\varphi\left(\cup_{k=0}^{\infty}\{u\in \C:\text{Re}(u)>\rho_k\}\times\mathbb{D}_k\right) \subset P(\mathcal{V}).$$

Therefore
$$\cup_{n=0}^{\infty} \cup_{k=0}^{\infty}(\chi^{-1})^{\circ n}\circ\varphi\left(\{u\in \C:\text{Re}(u)>\rho_k\}\times\mathbb{D}_k\right)\subseteq\mathcal Q(\Omega).$$

Equation \eqref{made}  shows that for every $k$ we can find $r_k\geq \rho_k$ such that
\[
\{u\in \C:\text{Re}(u-r_k)>|{\sf Im}(u)|\}\times\mathbb{D}_{k/2} \subseteq \varphi\left(\{u\in \C:\text{Re}(u)>\rho_k\}\times\mathbb{D}_k\right),
\]
hence
\[
\C^2=\cup_{n=0}^{\infty}\cup_{k=0}^{\infty} (\chi^{-1})^{\circ n}\left(\{u\in \C:\text{Re}(u-r_k)>|{\sf Im}(u)|\}\times\mathbb{D}_{k/2}\right)\subseteq \mathcal{Q}(\Omega),
\]
and we are done.
\end{proof}

Theorem \ref{main} now follows from Proposition \ref{prop1} and Proposition \ref{prop2}.

\section{An example}\label{last}

Using shears and overshears we can construct an explicit automorphism of the form (\ref{form}).  We first define automorphisms
\begin{align*}
F_1(z,w)&=(z,\lambda w+z),\\
F_2(z,w)&=(ze^w,w),\\
F_3(z,w)&=(z,w-z),\\
F_4(z,w)&=(ze^{-w},w),\\
F_5(z,w)&=(z,we^{z}),
\end{align*}
and finally
$$
F(z,w):=(F_5\circ F_4\circ F_3\circ F_2\circ F_1)(z,w).
$$
Quick computation shows that
$$
F(z,w)=\left(z+e^{\lambda w}z^2+O(z^3), \lambda w-z\sum_{k=2}^{\infty}\frac{\lambda^k}{k!}w^k +O(z^2) \right)
$$
and
$$
F^{-1}(z,w)=\left(z-e^{ w}z^2+O(z^3), \lambda^{-1} w+\lambda^{-1}z\sum_{k=2}^{\infty}\frac{1}{k!}w^k +O(z^2) \right).
$$
\bibliographystyle{amsplain}

\end{document}